\documentclass{birkjour}

\usepackage[utf8]{inputenc}
\usepackage[T1]{fontenc}
\usepackage[english]{babel}
\usepackage{mathtools,amsmath,amsfonts,amsthm,amssymb,mathrsfs}
\mathtoolsset{showonlyrefs}
\usepackage{enumitem}

\usepackage{graphicx,color}

\newtheorem{theorem}{Theorem}[section]

\newtheorem{lemma}{Lemma}[section]

\newtheorem{corollary}{Corollary}[section]
\newtheorem*{corollary*}{Corollary}
\theoremstyle{definition}
\newtheorem{remark}{Remark}[section]

\numberwithin{equation}{section}

\usepackage{geometry}
\begin{document}

\title[Systems of indefinite differential equations]{Multiplicity of clines for systems \\of indefinite differential equations arising \\ from a multilocus population genetics model}

\author[G.~Feltrin]{Guglielmo Feltrin}

\address{
Department of Mathematics, Computer Science and Physics, University of Udine\\
Via delle Scienze 206, 33100 Udine, Italy}

\email{guglielmo.feltrin@uniud.it}

\author[P.~Gidoni]{Paolo Gidoni}

\address{
Czech Academy of Sciences, Institute of Information Theory and Automation (UTIA)
Department of Decision-Making Theory\\
Pod vod\'{a}renskou ve\v{z}\'{i} 4, CZ-182 08, Prague 8, Czech Republic}

\email{gidoni@utia.cas.cz}

\thanks{Work written under the auspices of the 
Grup\-po Na\-zio\-na\-le per l'Anali\-si Ma\-te\-ma\-ti\-ca, la Pro\-ba\-bi\-li\-t\`{a} e le lo\-ro
Appli\-ca\-zio\-ni (GNAMPA) of the Isti\-tu\-to Na\-zio\-na\-le di Al\-ta Ma\-te\-ma\-ti\-ca (INdAM).
\\
\textbf{Preprint -- September 2019}} 



\date{}

\dedicatory{}

\begin{abstract}
We investigate sufficient conditions for the presence of coexistence states for different genotypes in a diploid diallelic population with dominance distributed on a heterogeneous habitat, considering also the interaction between genes at multiple loci. In mathematical terms, this corresponds to the study of the Neumann boundary value problem 
\begin{equation*}
\begin{cases}
\, p_{1}''+\lambda_{1} w_{1}(x,p_{2}) f_{1}(p_{1}) = 0, &\text{in $\Omega$,} \\
\, p_{2}''+\lambda_{2} w_{2}(x,p_{1}) f_{2}(p_{2}) = 0, &\text{in $\Omega$,} \\
\, p_{1}'=p_{2}'=0, &\text{on $\partial\Omega$,} 
\end{cases}
\end{equation*}
where the coupling-weights $w_{i}$ are sign-changing in the first variable, and the nonlinearities $f_{i}\colon\mathopen{[}0,1\mathclose{]}\to\mathopen{[}0,+\infty\mathclose{[}$ satisfy $f_{i}(0)=f_{i}(1)=0$, $f_{i}(s)>0$ for all $s\in\mathopen{]}0,1\mathclose{[}$, and a superlinear growth condition at zero. Using a topological degree approach, we prove existence of $2^{N}$ positive fully nontrivial solutions when the real positive parameters $\lambda_{1}$ and $\lambda_{2}$ are sufficiently large.
\end{abstract}

\maketitle

\section{Introduction}\label{section-1}

Starting from the Eighties, a great deal of attention has been devoted to boundary value problems associated with differential equations of the form
\begin{equation}\label{eq:PDE_indweight}
-\Delta p = w(x) f(p),
\end{equation}
where $w$ changes sign in the domain. Following a terminology popularised in \cite{HeKa-80}, such problems are referred to as \emph{indefinite weight} problems.
Many issues connected to this family of problems, like existence, uniqueness, multiplicity and stability of solutions for \eqref{eq:PDE_indweight} have been largely investigates; among many others, we mention \cite{AlTa-93,AmLG-98,BaPoTe-88,BeCDNi-94,Bu-76,Ra-7374}, and we refer to the introduction in \cite{Fe-18book} for a more complete bibliography and to \cite{Ac-09} for a survey illustrating some examples of models in physics and ecology, where the search of stationary solutions of parabolic equations strongly motivates the study of equation \eqref{eq:PDE_indweight}, also in the ODE case.

More recently, a series of papers (e.g.~\cite{BoFeZa-16,BoFeZa-18tams,FeSo-18non,FeZa-15jde,GaHaZa-03asmm,LGTeZa-14}) has shown the effectiveness of topological methods, such as degree theory and shooting techniques, to study such issues in an ODE setting. More precisely, by analysing the nodal behaviour of the indefinite weight, it has been provided precise multiplicity results for positive solutions of the equation
\begin{equation*}
p'' + w(x) f(p) = 0 
\end{equation*}
considering several types of boundary conditions and several types of nonlinearities $f$ all characterized by a superlinear behaviour at zero.
As we discuss briefly below and more extensively in Section~\ref{section-5.1}, such framework is relevant to justify the coexistence of two competing alleles with dominance in a diploid population.

The main purpose of this paper is to illustrate how such an approach can be successfully extended to the case of systems of indefinite equations.
Indeed, their flexibility is one of the advantages of the application of topological methods. The employment of topological tools to generalize existence results for ODEs to systems of differential equations obtained by suitable coupling of the original equation has been recently applied in various frameworks, for instance in \cite{BoDa-14,FonGarGid16,FonGid19,FonSfe16}.
We remark that such results are not necessarily restricted to small perturbations: indeed, as in our case, they apply also to larger suitable couplings.

In terms of modelling, our generalisation of the problem to the case of systems is justified by considering a nontrivial interaction between genes at multiple loci.
Let us therefore consider a diploid population. We assume that at one locus we have two possible alleles $A$ and $a$, with complete dominance of the allele $A$; whereas at a second locus we have, analogously, two possible alleles $B$ and $b$, with complete dominance of the allele $B$. We denote with $p,q\in \mathopen{[}0,1\mathclose{]}$, respectively, the frequencies of the alleles $a$ and $b$ in the population. We further assume that the population is at equilibrium, with random mating and linkage equilibrium between the two loci. The population is distribute on a heterogeneous bounded habitat $\Omega\subseteq \mathbb{R}$. As we derive rigorously in Section~\ref{section-5}, the fitness of the allele $a$ at a given place $x$ can be expressed as the product of a term $f(p)=p ^{2}(1-p)$, depending on the frequency of the allele and accounting for the dominance of $A$, and of a Carath\'{e}odory function $\lambda w_{p}(x,q)$ which describes how favourable is the habitat to the allele $a$ at $x$, and depends also on the genotype distribution of the population in $x$ with respect to the other locus, expressed through the dependence of the frequency $q$. Such structure is directly calculated by the space-dependent fitnesses of the four possible phenotypes.
We assume that the habitat presents a region $I_{p}\subseteq \Omega$ where it is favourable to the allele $a$ for all the possible situations at the other locus (namely $w_{p}(x,q)>0$ for every $x\in I_{p}$), but that in general the environment is hostile to the allele $a$ (namely $w_{p}(x,q)<\beta(x)$ with $\int_\Omega \beta(x)\,\mathrm{d} x<0$).
The coefficient $\lambda$ measures the intensity of the competition between the alleles compared to the velocity of diffusion, meaning that a larger $\lambda$ corresponds to a larger advantage of one of the allele where it is favoured by the environment.
An analogous characterisation holds for the allele $b$.
We also assume a random dispersal of the population within the habitat, whose intensity is regulated by a function $\kappa(x)>\kappa_{0}>0$.

The search for steady states for the population corresponds to solving the system
\begin{equation}\label{sys-intro-pq}
\begin{cases}
\, \kappa(x) p''+\lambda w_{p}(x,q) f(p) = 0, &\text{in $\Omega$,} \\
\, \kappa(x) q''+\lambda w_{q}(x,p) f(q) = 0, &\text{in $\Omega$,} \\
\, p'=q'=0, &\text{on $\partial\Omega$.} 
\end{cases}
\end{equation}
In particular we are interested in finding steady states where both alleles coexist at each locus, namely such that $0<p(x)<1$ and $0<q(x)<1$ for all $x\in \Omega$. We refer to such solutions of \eqref{sys-intro-pq} as \emph{fully nontrivial solutions}. Such stationary solutions are usually also called \emph{clines}, since they describe the gradual change in space of the frequency of an allele.

As a corollary of our main result (Theorem~\ref{th-main}) we obtain the following necessary condition for the existence of fully nontrivial stationary states (cf.~Corollary~\ref{cor-intro}).

\begin{corollary*}
There exists $\lambda^{*}>0$ such that for every $\lambda >\lambda^{*}$ there exist at least four fully nontrivial solutions to \eqref{sys-intro-pq}.
\end{corollary*}

Beyond this motivating application, since the problem has an intrinsic mathematical interest, we remark that our main result deals with a more general family of systems. Indeed, as we discuss in Section~\ref{section-2}, we consider functions $f$ satisfying
\begin{equation*}
f(0)=f(1)=0, \qquad f(s)>0, \; \text{for all $s\in\mathopen{]}0,1\mathclose{[}$,}
\end{equation*}
and having a superlinear growth at zero, namely $f'(0)=0$. Moreover our result applies to an arbitrarily large number of coupled equations.

\medskip

The plan of the paper is as follows. In Section~\ref{section-2} we introduce the mathematical framework and state our main result. Section~\ref{section-3} is devoted to illustrate the abstract degree setting which is employed in the proof of Theorem~\ref{th-main}, given in Section~\ref{section-4}. The paper ends with Section~\ref{section-5}, where we derive the model described above.

\section{Notation, hypotheses and main result}\label{section-2}

We now introduce our working framework, which is more general than the one discussed in the Introduction. In particular we consider a system of $N$ ordinary differential equations, where $N\geq 2$ is an integer. In this section, we list all the hypotheses needed and state our main result.

Let $\Omega = \mathopen{]}\underline{\omega},\overline{\omega}\mathclose{[}$ be an open and bounded interval in $\mathbb{R}$.
For $i=1,\ldots,N$, let $w_{i}\colon\Omega\times\mathbb{R}^{N-1}\to\mathbb{R}$ be an $L^{1}$-Carath\'{e}odory function satisfying
\begin{enumerate}[label=\textup{$(w_{i,*})$}]
\item there exists a (non-empty) closed interval $I_{i}\subseteq\overline{\Omega}$ such that for all $\hat{\xi}\in\mathopen{[}0,1\mathclose{]}^{N-1}$,
\begin{equation*}
w_{i}(x,\hat{\xi}) \geq 0, \quad w_{i}(\cdot,\hat{\xi})\not\equiv0,\quad \text{on $I_{i}$,}
\end{equation*}
and there exist $\alpha_{i},\beta_{i}\in L^{1}(\Omega)$ with
\begin{equation*}
\text{$\alpha_{i}(x)\geq 0$ and $\alpha_{i}\not\equiv0$ on $I_{i}$,} \quad \int_{\Omega} \beta_{i}(x) \,\mathrm{d}x < 0,
\end{equation*}
such that
\begin{equation*}
\alpha_{i}(x) \leq w_{i}(x,\hat{\xi}) \leq \beta_{i}(x), \quad \text{for a.e.~$x\in\Omega$, for all $\hat{\xi}\in\mathopen{[}0,1\mathclose{]}^{N-1}$.}
\end{equation*}
\label{cond:wi*}
\end{enumerate}
Here, the notation $g\not\equiv 0$ means that the function $g$ is not identically zero.
The first part of hypothesis \ref{cond:wi*} is a sign-condition which states that $w_{i}$ is positive on an interval of $\Omega$ with non-zero measure, while the second part implies that $w_{i}$ is $L^{1}$-uniformly bounded in the second variable and $\beta_{i}$ is negative somewhere (and so also the $w_i$). We notice that the weights $w_{i}$ can change sign infinitely many times.

For $i=1,\ldots,N$, let $f_{i}\colon\mathopen{[}0,1\mathclose{]}\to\mathopen{[}0,+\infty\mathclose{[}$ be a continuously differentiable function such that
\begin{enumerate}[label=\textup{$(f_{i,*})$}]
\item $f_{i}(0)=f_{i}(1)=0$, $f_{i}(s)>0$, for all $s\in\mathopen{]}0,1\mathclose{[}$. \label{cond:fi*}
\end{enumerate}
Moreover, we assume a superlinear growth condition at zero, that is
\begin{enumerate}[label=\textup{$(f_{i,0})$}]
\item $f_{i}'(0)=0$.
\label{cond:fi0}
\end{enumerate}

We now introduce a useful notation. Let $i\in\{1,\ldots,N\}$ and $J\subseteq\mathbb{R}$. Given a vector field $p=(p_{1},\ldots,p_{N})\colon J\to\mathbb{R}^{N}$ we denote by
\begin{equation*}
\hat{p}^{i}(x) = \bigl{(} p_{1}(x),\ldots, p_{i-1}(x), p_{i+1}(x), \ldots, p_{N}(x) \bigr{)}\in\mathbb{R}^{N-1}, \quad x\in J,
\end{equation*}
the vector of length $N-1$ obtained from $p$ by removing the $i$-th component.

We deal with the following system of $N$ differential equations together with Neumann boundary conditions
\begin{equation}\label{system-S}
\begin{cases}
\, p_{i}''+\lambda_{i} w_{i}(x,\hat{p}^{i}) f_{i}(p_{i}) = 0, &\text{in $\Omega$,} \\
\, p_{i}'=0, &\text{on $\partial\Omega$,} \\
\, i=1,\ldots,N.
\end{cases}
\tag{$\mathcal{S}$}
\end{equation}
Solutions to \eqref{system-S} are meant in the Carath\'{e}odory sense, that is, a function $p=(p_{1},\ldots,p_{N})$, with $p_{i} \colon \overline{\Omega} \to \mathopen{[}0,1\mathclose{]}$, is a solution of \eqref{system-S} if $p_{i}$ is 
continuously differentiable in $\overline{\Omega}$ and $p_{i}'$ is absolutely continuous in $\Omega$, the differential equations in \eqref{system-S} are satisfied almost everywhere and $p_{i}'(x)=0$ for all $x\in\partial\Omega$, for every $i=1,\ldots,N$. We stress that, since $f_{i}$ is defined on $\mathopen{[}0,1\mathclose{]}$, we implicitly assume that $0\leq p_{i}(x) \leq 1$ for all $x\in\overline{\Omega}$, for every $i=1,\ldots,N$.

Hypotheses \ref{cond:fi*} ensures that $p=(p_{1},\ldots,p_{N})$ with either $p_{i}\equiv0$ or $p_{i}\equiv1$ for every $i=1,\ldots,N$, is a (constant) \textit{trivial} solution of \eqref{system-S}. In the sequel, we call \textit{semitrivial} a solution with $p_{i}\not\equiv0$ and $p_{i}\not\equiv1$ for some $i=1,\ldots,N$, and \textit{fully nontrivial} if $p_{i}\not\equiv0$ and $p_{i}\not\equiv1$ for every $i=1,\ldots,N$.

Our main result is the following.

\begin{theorem}\label{th-main}
For $i=1,\ldots,N$, let $w_{i}\colon\Omega\times\mathbb{R}^{N-1}\to\mathbb{R}$ be an $L^{1}$-Carath\'{e}odory function satisfying \ref{cond:wi*}. For $i=1,\ldots,N$, let $f_{i}\colon\mathopen{[}0,1\mathclose{]}\to\mathopen{[}0,+\infty\mathclose{[}$ be a continuously differentiable function satisfying \ref{cond:fi*} and \ref{cond:fi0}. 
Then, there exists $\lambda^{*}>0$ such that the following holds: if $\lambda_{i}>\lambda^{*}$ for every $i=1,\ldots,N$, then there exist at least $2^{N}$ fully nontrivial solutions to \eqref{system-S}. Moreover, there exists at least $4^{N}-2^{N}$ semitrivial solutions to \eqref{system-S}.
\end{theorem}

\section{Abstract degree setting}\label{section-3}

This section presents the abstract setting of the coincidence degree in the framework of system \eqref{system-S} and  two crucial lemmas for the computation of the degree. In the following, for $i=1,\ldots,N$, we assume that $\lambda_{i}>0$, $w_{i}$ is an $L^{1}$-Carath\'{e}odory function satisfying \ref{cond:wi*}, $f_{i}\in\mathcal{C}(\mathopen{[}0,1\mathclose{]})$ satisfying \ref{cond:fi*}.

As a first step, we extend the nonlinearities contained in \eqref{system-S} to the whole real line, by defining the $L^{1}$-Carath\'{e}odory function\begin{equation*}
h(x,\xi) = (h_{1}(x,\xi), \ldots, h_{N}(x,\xi)), \quad \text{$x\in\Omega$, $\xi=(\xi_{1},\ldots,\xi_{N})\in\mathbb{R}^{N}$,}
\end{equation*}
where, for $i=1,\ldots,N$, we set
\begin{equation*}
h_{i}(x,\xi) = 
\begin{cases}
\, -\xi_{i}, & \text{if $\xi_{i}\leq 0$,} \\
\, \lambda_{i} w_{i}(x,\hat{\xi}^{i}) f_{i}(\xi_{i}), & \text{if $0\leq \xi_{i} \leq 1$,} \\
\, 1-\xi_{i}, & \text{if $\xi_{i}\geq 1$,}
\end{cases}
\quad \text{$x\in\Omega$, $\xi\in\mathbb{R}^{N}$,}
\end{equation*}
with $\hat{\xi}^{i}=(\xi_{1},\ldots,\xi_{i-1},\xi_{i+1},\ldots,\xi_{N})\in\mathbb{R}^{N-1}$.

In this manner, as a consequence of the weak maximum principle (based on a convexity argument), every solution $p$ of
\begin{equation}\label{system-H}
\begin{cases}
\, p_{i}''+\lambda_{i} h_{i}(x,p) = 0, &\text{in $\Omega$,} \\
\, p_{i}'=0, &\text{on $\partial\Omega$,} \\
\, i=1,\ldots,N,
\end{cases}
\tag{$\mathcal{H}$}
\end{equation}
satisfies $0\leq p_{i}(x)\leq 1$, for all $x\in\overline{\Omega}$ (for $i=1,\ldots,N$), and thus solves \eqref{system-S}.

\subsection{Operatorial formulation}\label{section-3.1}

We plan to exploit the coincidence degree theory introduced and developed by J.~Mawhin (cf.~\cite{GaMa-77,Ma-79,Ma-93}). First of all, we aim to write system \eqref{system-H} as an operatorial equation of the form
\begin{equation}\label{eq-coinc}
Lp = Np,\quad p\in \mathrm{dom}\,L.
\end{equation}
As quite standard in the framework of Neumann boundary value problems, we define $L\colon \mathrm{dom}\,L \to L^{1}(\Omega,\mathbb{R}^{N})$ as the linear differential operator 
\begin{equation*}
(Lp)(x)= -p''(x), \quad x\in\Omega,
\end{equation*}
with
\begin{equation*}
\mathrm{dom}\,L = \bigl{\{} p \in W^{2,1}(\overline{\Omega},\mathbb{R}^{N}) \colon p_{i}'(\underline{\omega}) = p_{i}'(\overline{\omega})=0, \; i=1,\ldots,N \bigr{\}} \subseteq \mathcal{C}(\overline{\Omega},\mathbb{R}^{N}) ,
\end{equation*}
and
$N \colon \mathcal{C}(\overline{\Omega},\mathbb{R}^{N}) \to L^{1}(\Omega,\mathbb{R}^{N})$ as the Nemytskii operator induced by the function $h$, precisely
\begin{equation*}
(Np)(x) = h(x,p(x)), \quad x\in\Omega.
\end{equation*}

We now recall the main features of coincidence degree theory which will be crucial in the sequel.
Let
\begin{equation*}
X = \prod_{i=1}^{N} X_{i} = \mathcal{C}(\overline{\Omega},\mathbb{R}^{N}) \quad \text{ and } \quad Z = \prod_{i=1}^{N} Z_{i} = L^{1}(\Omega,\mathbb{R}^{N}),
\end{equation*}
where, for $i=1,\ldots,N$, $X_{i} = \mathcal{C}(\overline{\Omega})$ is the Banach space of continuous functions $p_{i} \colon \overline{\Omega} \to \mathbb{R}$,
endowed with the $\sup$-norm $\|p_{i}\|_{\infty} = \max_{x\in \overline{\Omega}} \lvert p_{i}(x)\rvert$,
and $Z_{i}=L^{1}(\Omega)$ is the Banach space of Lebesgue integrable functions $z_{i} \colon \Omega \to \mathbb{R}$, endowed with the $L^{1}$-norm $\|z_{i}\|_{L^{1}}= \int_{\Omega} \lvert z_{i}(x) \rvert\,\mathrm{d}x$.
The spaces $X$ and $Z$ are endowed with the standard norms.

For $i=1,\ldots,N$, let $L_{i}\colon \mathrm{dom}\,L_{i} \to Z_{i}$ be the $i$-th component of $L =(L_{1},\ldots,L_{N})$, that is $\mathrm{dom}\,L_{i} = \bigl{\{} p_{i} \in W^{2,1}(\overline{\Omega}) \colon p_{i}'(\underline{\omega}) = p_{i}'(\overline{\omega})=0 \bigr{\}} \subseteq X_{i}$ and $L_{i}p_{i}= -p''_{i}$.
We observe that $\ker L_{i} \equiv \mathbb{R}$ and $\mathrm{coker} L_{i} \equiv \mathbb{R}$ are made up of constant functions and 
\begin{equation*}
\mathrm{Im}\,L_{i} = \biggl{\{} z_{i}\in Z_{i} \colon \int_{\Omega}z_{i}(x) \,\mathrm{d}x= 0 \biggr{\}}.
\end{equation*}
We also define the projectors $P\colon X\to \ker L$ and $Q\colon Z \to \mathrm{coker}\,L$ with components
\begin{equation*}
P_{i} p_{i} = \dfrac{1}{|\Omega|}\int_{\Omega}p_{i}(x) \,\mathrm{d}x, \quad Q_{i} z_{i} = \dfrac{1}{|\Omega|}\int_{\Omega}z_{i}(x) \,\mathrm{d}x, \quad i=1,\ldots,N,
\end{equation*}
and the map $K=(K_{1},\ldots,K_{N})\colon \mathrm{Im}\,L\to\mathrm{dom}\,L\cap\ker P$ as the right inverse of $L$, namely, for $i=1,\ldots,N$, $K_{i}$ associates to $v\in L^{1}(\Omega)$ with $\int_{\Omega} v(x)\,\mathrm{d}x =0$ the unique solution of $p_{i}''+v(x) =0$ with $\int_{\Omega} p_{i}(x)\, \mathrm{d}x=0$.

With the above position, one can show that $p$ is a solution of \eqref{system-H} if and only if $p$ satisfy the coincidence equation \eqref{eq-coinc}.
From Mawhin's coincidence degree theory, it is standard to prove that \eqref{eq-coinc} is equivalent to the fixed point problem
\begin{equation*}
p = \Phi (p), \quad p\in X,
\end{equation*}
where $\Phi =(\Phi_{1},\ldots,\Phi_{N}) \colon X \to X$ is the completely continuous operator
\begin{equation*}
\Phi(p) = Pp + JQNp + K(\mathrm{Id}_{Z}-Q)Np, \quad p\in X,
\end{equation*}
where $J\colon \mathrm{coker}\,L\to\ker L$ is the identity map.

\medskip

Given an open and bounded set $\mathcal{B}\subseteq X$ such that
\begin{equation*}
L p \neq N p, \quad \text{for all $p\in\partial \mathcal{B}\cap\mathrm{dom}\,L$,}
\end{equation*}
the \textit{coincidence degree $\mathrm{D}_{L}(L-N,\mathcal{B})$ of $L$ and $N$ in $\mathcal{B}$} is defined as
\begin{equation*}
\mathrm{D}_{L}(L-N, \mathcal{B}) = \textrm{deg}_{\textrm{LS}}(\mathrm{Id}_{X}-\Phi, \mathcal{B}, 0),
\end{equation*}
where ``$\textrm{deg}_{\textrm{LS}}$'' denotes the Leray--Schauder topological degree.
The coincidence degree satisfies all the usual properties of Brouwer's and Leray--Schauder degrees, like additivity/excision and homotopic invariance properties. In particular, equation \eqref{eq-coinc} has at least a solution in $\mathcal{B}$ if $\mathrm{D}_{L}(L-N,\mathcal{B})\neq0$.
For a more detailed discussion, we refer to \cite{GaMa-77,Ma-79,Ma-93}.

\subsection{Technical degree lemmas}\label{section-3.2}

We present two lemmas for the computation of the coincidence degree on open bounded sets of the form
\begin{equation*}
\mathcal{B}_{d}^{\mathcal{I}} = \prod_{i=1}^{N} B_{i}, \quad \text{with $d=(d_{1},\ldots,d_{N})\in\mathopen{]}0,1\mathclose{[}^{N}$,}
\end{equation*}
where
\begin{equation*}
B_{i}= 
\begin{cases}
\, B(0,d_{i}), & \text{if $i\notin\mathcal{I}$,} \\
\, U_{d_{i},I_{i}} = \biggl{\{} p_{i} \in X_{i} \colon \|p_{i}\|_{\infty} <1, \; \displaystyle \max_{x\in I_{i}} |p_{i}(x)| < d_{i} \biggr{\}},
& \text{if $i\in\mathcal{I}$,}
\end{cases}
\end{equation*}
denoting with $B(0,d_i)$ the open ball in $X_i$ with center $0$ and radius $d_i$, and recalling the interval $I_{i}$ introduced in hypothesis \ref{cond:wi*}.

The first lemma presents sufficient conditions for zero degree.

\begin{lemma}\label{lem-deg0}
For $i=1,\ldots,N$, let $w_{i}\colon\Omega\times\mathbb{R}^{N-1}\to\mathbb{R}$ be an $L^{1}$-Carath\'{e}odory function satisfying \ref{cond:wi*}. For $i=1,\ldots,N$, let $f_{i}\colon\mathopen{[}0,1\mathclose{]}\to\mathopen{[}0,+\infty\mathclose{[}$ be a continuous function satisfying \ref{cond:fi*}. 
Let $\mathcal{I}\subseteq\{1,\ldots,N\}$ with $\mathcal{I}\neq\emptyset$, $d=(d_{1},\ldots,d_{N})\in\mathopen{]}0,1\mathclose{[}^{N}$, and assume that there exists $v = (v_{1},\ldots,v_{N})\in L^{1}(\Omega,\mathbb{R}^{N})$, with $v\not\equiv0$, such that the following two properties hold:
\begin{enumerate}[label=\textup{$(H_{\arabic*})$}]
\item
If $\mu \geq 0$ and $p$ is a solution of
\begin{equation}\label{system-Fa}
\begin{cases}
\, p_{i}''+ \lambda_{i} w_{i}(x,\hat{p}^{i}) f_{i}(p_{i}) + \mu v_{i}(x) = 0, &\text{in $\Omega$,} \\
\, p_{i}'=0, &\text{on $\partial\Omega$,} \\
\, i=1,\ldots,N,
\end{cases}
\end{equation}
then 
\begin{align*}
& \|p_{i}\|_{\infty} \neq d_{i}, \quad \text{for all $i\in\{1,\ldots,N\}\setminus \mathcal{I}$,} \\
& \max_{x\in I_{i}} |p_{i}(x)| \neq d_{i}, \quad \text{for all $i\in\mathcal{I}$.}
\end{align*}
\label{cond-H1}
\item
There exists $\mu_{0} \geq 0$ such that problem \eqref{system-Fa}, with $\mu=\mu_{0}$, has no solution $p$ with 
\begin{align*}
& \|p_{i}\|_{\infty} \leq d_{i}, \quad \text{for all $i\in\{1,\ldots,N\}\setminus \mathcal{I}$,} \\
& \max_{x\in I_{i}} |p_{i}(x)| \leq d_{i}, \quad \text{for all $i\in\mathcal{I}$.}
\end{align*}
\label{cond-H2}
\end{enumerate}
Then, it holds that $\mathrm{D}_{L}(L-N,\mathcal{B}_{d}^{\mathcal{I}}) = 0$.
\end{lemma}

\begin{proof}
Let $\mathcal{I}\subseteq\{1,\ldots,N\}$ with $\mathcal{I}\neq\emptyset$ and $d=(d_{1},\ldots,d_{N})\in\mathopen{]}0,1\mathclose{[}^{N}$. Let $v = (v_{1},\ldots,v_{N})\in L^{1}(\Omega,\mathbb{R}^{N})$, with $v\not\equiv0$, be such that hypotheses \ref{cond-H1} and \ref{cond-H2} hold.

We study the parameter-dependent problem
\begin{equation}\label{system-HFa}
\begin{cases}
\, p_{i}'' + h_{i}(x,p) + \mu v_{i}(x) = 0, &\text{in $\Omega$,} \\
\, p_{i}'=0, &\text{on $\partial\Omega$,} \\
\, i=1,\ldots,N,
\end{cases}
\end{equation}
for $\mu \geq 0$, which can be equivalently written as a coincidence equation in the space $X$
\begin{equation*}
Lp = Np + \mu v, \quad p \in \mathrm{dom}\,L.
\end{equation*}
We notice that if $p=(p_{1},\ldots,p_{N})$ is a solution of \eqref{system-HFa}, then the weak maximum principle ensures that $0 \leq p_{i}(x)\leq 1$ for all $x\in \overline{\Omega}$, for every $i=1,\ldots,N$, and, indeed, $p$ solves \eqref{system-Fa}.

We first observe that $\|p_{i}\|_{\infty}<1$ for every $p_{i}\in\partial U_{d_{i},I_{i}}$, for every $i\in\mathcal{I}$, due the uniqueness of the solution of the Cauchy problem associated with the $i$-th equation in \eqref{system-HFa}. Therefore, by condition \ref{cond-H1}, the coincidence degree $\mathrm{D}_{L}(L-N - \mu v,\mathcal{B}_{d}^{\mathcal{I}})$ is well-defined for every $\mu\geq0$. Now, using $\mu$ as a homotopy parameter and exploiting the homotopy invariance property of the coincidence degree, we have that
\begin{equation*}
\mathrm{D}_{L}(L-N,\mathcal{B}_{d}^{\mathcal{I}}) = \mathrm{D}_{L}(L-N - \mu_{0} v,\mathcal{B}_{d}^{\mathcal{I}}) = 0,
\end{equation*}
where the last equality follows from \ref{cond-H2}.
This concludes the proof.
\end{proof}

The second lemma states a sufficient condition for non-zero degree in
\begin{equation*}
\mathcal{B}_{d}^{\emptyset} =
 B(0,d_{1}) \times \dots \times B(0,d_{N}) \subseteq X, \quad \text{with $d=(d_{1},\ldots,d_{N})\in\mathopen{]}0,1\mathclose{[}^{N}$,}
\end{equation*}
namely in a Cartesian product of open balls in $\mathcal{C}(\overline{\Omega})$.
We recall that
\begin{equation*}
\partial \mathcal{B}_{d}^{\emptyset} = \bigcup_{i=1}^{N} \Bigl{(} \overline{B(0,d_{1})}\times \dots \times \overline{B(0,d_{i-1})} \times \partial B(0,d_{i}) \times \overline{B(0,d_{i+1})}\dots \times \overline{B(0,d_{N})} \Bigr{)}.
\end{equation*}

\begin{lemma}\label{lem-deg1}
For $i=1,\ldots,N$, let $w_{i}\colon\Omega\times\mathbb{R}^{N-1}\to\mathbb{R}$ be an $L^{1}$-Carath\'{e}odory function satisfying \ref{cond:wi*}. For $i=1,\ldots,N$, let $f_{i}\colon\mathopen{[}0,1\mathclose{]}\to\mathopen{[}0,+\infty\mathclose{[}$ be a continuous function satisfying \ref{cond:fi*}. 
Let $d=(d_{1},\ldots,d_{N})\in\mathopen{]}0,1\mathclose{[}^{N}$ and assume that the following property holds.
\begin{enumerate}[label=\textup{$(H_{3})$}]
\item If $\vartheta\in \mathopen{]}0,1\mathclose{]}$ and $p$ is a solution of
\begin{equation}\label{system-S_theta}
\begin{cases}
\, p_{i}''+\vartheta \lambda_{i} w_{i}(x,\hat{p}^{i}) f_{i}(p_{i}) = 0, &\text{in $\Omega$,} \\
\, p_{i}'=0, &\text{on $\partial\Omega$,} \\
\, i=1,\ldots,N,
\end{cases}
\end{equation}
then $\|p_{i}\|_{\infty} \neq d_{i}$, for all $i\in\{1,\ldots,N\}$.
\label{cond-H3}
\end{enumerate}
Then, it holds that $\mathrm{D}_{L}(L-N,\mathcal{B}_{d}^{\emptyset}) = 1$.
\end{lemma}

\begin{proof}
Let $d=(d_{1},\ldots,d_{N})\in\mathopen{]}0,1\mathclose{[}^{N}$. As a first step, we show that
\begin{equation*}
L p \neq \vartheta N p, \quad \text{for all $p\in\partial \mathcal{B}_{d}^{\emptyset}\cap\mathrm{dom}\,L$ and $\vartheta\in\mathopen{]}0,1\mathclose{]}$.}
\end{equation*}
Indeed, if $p\in \overline{\mathcal{B}_{d}^{\emptyset}}\cap\mathrm{dom}\,L$ is a solution to $L p = \vartheta N p$ for some $\vartheta\in\mathopen{]}0,1\mathclose{]}$, then $p$ solves
\begin{equation*}
\begin{cases}
\, p_{i}''+\vartheta h_{i}(x,p) = 0, &\text{in $\Omega$,} \\
\, p_{i}'=0, &\text{on $\partial\Omega$,} \\
\, i=1,\ldots,N.
\end{cases}
\end{equation*}
By the weak maximum principle, we find that $0 \leq p_{i}(x)\leq 1$ for all $x\in \overline{\Omega}$ and, indeed, $p$ solves \eqref{system-S_theta}. By condition \ref{cond-H3}, we deduce that $p\notin\partial \mathcal{B}_{d}^{\emptyset}$ and hence the claim is proved.

As a second step, we deal with the case $\vartheta=0$. We consider the operator $QN$ in $\ker L \equiv \mathbb{R}^{N}$, namely
\begin{equation*}
QNp = \dfrac{1}{|\Omega|}\int_{\Omega} h(x,\xi) \, \mathrm{d}x, \quad \text{$p\equiv \xi\in\mathbb{R}^{N}$.}
\end{equation*}
We introduce the map $h^{\#} =(h_{1}^{\#}, \ldots, h_{N}^{\#}) \colon \mathbb{R}^{N} \to \mathbb{R}^{N}$ defined, for $i=1,\ldots,N$, by
\begin{align*}
h_{i}^{\#}(\xi) &= \dfrac{1}{|\Omega|} \int_{\Omega} h_{i}(x,\xi) \,\mathrm{d}x
 =
\begin{cases}
\, -\xi_{i}, & \text{if $\xi_{i} \leq 0$,} \\
\, \lambda_{i} \biggl{(}\dfrac{1}{|\Omega|} \displaystyle \int_{\Omega} w_{i}(x,\hat{\xi}^{i}) \,\mathrm{d}x \biggr{)} f_{i}(\xi_{i}), & \text{if $0 \leq \xi_{i} \leq 1$,} \\
\, 1-\xi_{i}, & \text{if $\xi_{i} \geq 1$.}
\end{cases}
\end{align*}
Hence $QN \xi = h^{\#}(\xi)$ for all $\xi\in\mathbb{R}^{N}$.
By condition \ref{cond:wi*}, for $i=1,\ldots,N$, we have
\begin{equation*}
\int_{\Omega} w_{i}(x,\hat{\xi}) \,\mathrm{d}x \leq \int_{\Omega} \beta_{i}(x) \,\mathrm{d}x < 0, \quad \text{for all $\hat{\xi}\in\mathopen{[}0,1\mathclose{]}^{N-1}$,}
\end{equation*}
and thus
\begin{equation}\label{eq-hs}
\langle h^{\#}(\xi),\xi \rangle <0, \quad \text{for all $\xi\in\mathbb{R}^{N}$ with $\xi_{i}\neq 0$, $\xi_{i}\neq 1$, for every $i=1,\ldots,N$.}
\end{equation}
As a direct consequence, $h^{\#}$ has no zeros on 
\begin{align*}
\partial(\mathcal{B}_{d}^{\emptyset}\cap\mathbb{R}^{N}) &= \bigcup_{i=1}^{N} \Bigl{(} \mathopen{[}-d_{1},d_{1}\mathclose{]} \times \dots \times \mathopen{[}-d_{i-1},d_{i-1}\mathclose{]} \times \{\pm d_{i} \} \times
\\ &\qquad \qquad \times\mathopen{[}-d_{i+1},d_{i+1}\mathclose{]} \times \dots \times \mathopen{[}-d_{N},d_{N}\mathclose{]} \bigr{)},
\end{align*}
and therefore $QNp\neq0$ for all $p\in\partial \mathcal{B}_{d}^{\emptyset}\cap \ker L$. 

An application of Mawhin's continuation theorem (see \cite{Ma-72}) ensures that
\begin{equation*}
\mathrm{D}_{L}(L-N,\mathcal{B}_{d}^{\emptyset}) = \mathrm{deg}_{\mathrm{B}}(-QN,\mathcal{B}_{d}^{\emptyset}\cap\ker L,0) = \mathrm{deg}_{\mathrm{B}}(-h^{\#},\mathcal{B}_{d}^{\emptyset}\cap\mathbb{R}^{N},0),
\end{equation*}
where ``$\textrm{deg}_{\textrm{B}}$'' denotes Brouwer's topological degree.

Finally, from \eqref{eq-hs} we deduce that Brouwer's degree of $-h^{\#}$ in $\mathcal{B}_{d}^{\emptyset}\cap\ker L$ is $1$. This fact can be straightforwardly verified by considering the homotopy given by a convex combination with the identity map.
\end{proof}

\section{Proof of the main result}\label{section-4}

The section presents the proof of Theorem~\ref{th-main}. First, we prove some preliminary lemmas. Then, we can give the proof which is based on the coincidence degree theory illustrated in Section~\ref{section-3}.

\subsection{Preliminary lemmas}\label{section-4.1}

The following three lemmas give some estimates for the solutions of the homotopic parameter dependent systems associated with \eqref{system-S} introduced in Section~\ref{section-3.2}.

\begin{lemma}\label{lem-rho}
For $i=1,\ldots,N$, let $w_{i}\colon\Omega\times\mathbb{R}^{N-1}\to\mathbb{R}$ be an $L^{1}$-Carath\'{e}odory function satisfying \ref{cond:wi*}. For $i=1,\ldots,N$, let $f_{i}\colon\mathopen{[}0,1\mathclose{]}\to\mathopen{[}0,+\infty\mathclose{[}$ be a continuous function satisfying \ref{cond:fi*}.
Let $v = (v_{1},\ldots,v_{N})\in L^{1}(\Omega,\mathbb{R}^{N})$, with $v_{i}\geq0$, for every $i=1,\ldots,N$.
Let $k\in\{1,\ldots,N\}$. For every $\rho\in\mathopen{]}0,1\mathclose{[}$, there exists $\lambda_{k}^{*} = \lambda_{k}^{*}(\rho) > 0$ such that, for every $\lambda_{k} > \lambda_{k}^{*}$, $\lambda_{i}>0$ for $i\neq k$ and $\mu \geq 0$, there is no solution $p=(p_{1},\ldots,p_{N})$ of
\begin{equation}\label{eq-lem-rho}
\begin{cases}
\, p_{i}''+\lambda_{i} w_{i}(x,\hat{p}^{i}) f_{i}(p_{i}) + \mu v_{i}(x) = 0, &\text{in $\Omega$,}\\
\, i=1,\ldots,N,
\end{cases}
\end{equation}
such that $\max_{x\in I_{k}} p_{k}(x) = \rho$.
\end{lemma}

\begin{proof}
For simplicity, for $i=1,\ldots,N$, we set $I_{i}=\mathopen{[}\sigma_{i},\tau_{i}\mathclose{]}$.
Let $k\in\{1,\ldots,N\}$ be fixed. Let $\varepsilon > 0$ with $\varepsilon < (\tau_{k}- \sigma_{k})/2$, and satisfying $\int_{\sigma_{k}+\varepsilon}^{\tau_{k}-\varepsilon} \alpha_{k}(x) \, \mathrm{d}x > 0$.

Let $\rho\in\mathopen{]}0,1\mathclose{[}$. We fix $\mu \geq 0$ and suppose by contradiction that $p$ is a solution of \eqref{eq-lem-rho} such that $\max_{x\in I_{k}} p_{k}(x) = \rho$. Hypothesis $(w_{k,*})$ ensures that $p_{k}$ is concave on $I_{k}$. As a consequence, we have
\begin{equation*}
p_{k}(x) \geq \dfrac{\rho}{\tau_{k}-\sigma_{k}}\min\bigl{\{}x-\sigma_{k},\tau_{k}-x\bigr{\}}\geq \dfrac{\varepsilon \rho}{|I_{k}|}, \quad \text{for all $x\in I_{k}$,}
\end{equation*}
(cf.~\cite[p.~420]{GaHaZa-03cpaa} for a similar estimate).
Moreover, since $p_{k}'$ is non-decreasing in $I_{k}$, by integrating we have that
\begin{equation*}
|p_{k}'(x)| \leq \dfrac{p_{k}(x)}{\varepsilon} \leq \dfrac{\rho}{\varepsilon}, \quad \text{for all $x\in \mathopen{[}\sigma_{k}+\varepsilon,\tau_{k}-\varepsilon\mathclose{]}$.}
\end{equation*}

We define
\begin{equation*}
\eta_{\varepsilon,\rho} = \min \Biggl{\{} f_{k}(s) \colon \dfrac{\varepsilon \rho}{|I_{k}|} \leq s \leq \rho \Biggr{\}}.
\end{equation*}
We integrate the $k$-th equation in \eqref{eq-lem-rho} on $\mathopen{[}\sigma_{k}+\varepsilon,\tau_{k}-\varepsilon\mathclose{]}$ and we use the above estimates to obtain
\begin{equation*}
\begin{aligned}
\lambda_{k} \eta_{\varepsilon,\rho} \int_{\sigma_{k}+\varepsilon}^{\tau_{k}-\varepsilon} \alpha_{k}(x) \,\mathrm{d}x
&\leq \lambda_{k} \int_{\sigma_{k}+\varepsilon}^{\tau_{k}-\varepsilon} w_{i}(x,\hat{p}^{k}(x))f_{k}(p_{k}(x)) \,\mathrm{d}x
\\ &= \int_{\sigma_{k}+\varepsilon}^{\tau_{k}-\varepsilon} (- p_{k}''(x) - \mu v_{k}(x)) \,\mathrm{d}x
\leq \int_{\sigma_{k}+\varepsilon}^{\tau_{k}-\varepsilon} - p_{k}''(x) \,\mathrm{d}x
\\ &= p_{k}'(\sigma_{k}+\varepsilon) - p_{k}'(\tau_{k}-\varepsilon) \leq \dfrac{2\rho}{\varepsilon}.
\end{aligned}
\end{equation*}
Finally, setting
\begin{equation*}
\lambda_{k}^{*} = \lambda_{k}^{*}(\rho) = \dfrac{2 \rho}{\varepsilon \eta_{\varepsilon,\rho} \int_{\sigma_{k}+\varepsilon}^{\tau_{k}-\varepsilon} \alpha_{k}(x) \,\mathrm{d}x}
\end{equation*}
and taking $\lambda_{k} > \lambda_{k}^{*}$, we conclude that there is no solution $p$ of \eqref{eq-lem-rho} with $\max_{x\in I_{k}} p_{k}(x) = \rho$.
\end{proof}

\begin{lemma}\label{lem-r}
For $i=1,\ldots,N$, let $w_{i}\colon\Omega\times\mathbb{R}^{N-1}\to\mathbb{R}$ be an $L^{1}$-Carath\'{e}odory function satisfying \ref{cond:wi*}. For $i=1,\ldots,N$, let $f_{i}\colon\mathopen{[}0,1\mathclose{]}\to\mathopen{[}0,+\infty\mathclose{[}$ be a continuously differentiable function satisfying \ref{cond:fi*} and \ref{cond:fi0}. 
Let $\lambda_{i}>0$, for $i=1,\ldots,N$. For every $k\in\{1,\ldots,N\}$ there exists $r_{k}\in\mathopen{]}0,1\mathclose{[}$ such that for all $\vartheta\in\mathopen{]}0,1\mathclose{]}$ every solution $p=(p_{1},\ldots,p_{N})$ of \eqref{system-S_theta} with $\|p_{k}\|_{\infty}\leq r_{k}$ satisfies $p_{k}(x)=0$ for all $x\in\overline{\Omega}$.
\end{lemma}

\begin{proof}
Let $k\in\{1,\ldots,N\}$ be fixed.
By contradiction, we assume that there exists a sequence $(p^{n})_{n}$ of solutions of \eqref{system-S_theta} for $\vartheta=\vartheta_{n}\in\mathopen{]}0,1\mathclose{]}$, with $p^{n}=(p_{1}^{n},\ldots,p_{N}^{n})$ such that $0<\|p_{k}^{n}\|_{\infty}=r^{n}\to 0$ as $n\to\infty$.
Setting
\begin{equation}\label{eq-zeta-n}
z_{i}^{n}(x) = \dfrac{(p_{i}^{n})'(x)}{\vartheta_{n} f_{i}(p_{i}^{n}(x))}, \quad \text{$x\in\Omega$,} \quad \text{for $i=1,\ldots,N$,}
\end{equation}
we deduce that
\begin{equation}\label{eq-zn}
\begin{aligned}
(z_{i}^{n})'(x) &= \dfrac{(p_{i}^{n})'' (x) \vartheta_{n} f_{i}(p_{i}^{n}(x)) - \vartheta_{n} f_{i}'(p_{i}^{n}(x)) ((p_{i}^{n})'(x))^{2}}{(\vartheta_{n})^{2}f_{i}(p_{i}^{n}(x))^{2}} 
\\ &= - \lambda_{i}w_{i}(x,\widehat{p^{n}}^{i}(x)) - \vartheta_{n} f_{i}'(p_{i}^{n}(x))(z_{i}^{n}(x))^{2},
\end{aligned}
\end{equation}
for almost every $x\in\Omega$, and
\begin{equation*}
z_{i}^{n}(\underline{\omega}) = z_{i}^{n}(\overline{\omega})=0,
\end{equation*}
since $p_{i}^{n}$ satisfies the Neumann boundary conditions on $\partial\Omega$.

Let $\gamma_{k}=\lambda_{k}\max\{|\alpha_{k}|,|\beta_{k}|\}$. We fix $M>\|\gamma_{k}\|_{L^{1}}$ and then $\delta\in\mathbb{R}$ with
\begin{equation}\label{eq-delta}
0 < \delta < \dfrac{M - \|\gamma_{k}\|_{L^{1}}}{|\Omega| M^{2}}.
\end{equation}
By condition $(f_{k,0})$ we can fix $\mu_{k}\in \mathopen{]}0,1\mathclose{[}$ such that
\begin{equation*}
|f_{k}'(s)|\leq \delta, \quad \text{for all $s\in\mathopen{[}0, \mu_{k}\mathclose{]}$.}
\end{equation*}
We notice that $0 < p_{k}^{n}(x) < \mu_{k}$ on $\overline{\Omega}$, for $n$ sufficiently large.
We claim that
\begin{equation}\label{eqz-ineq}
\|z_{k}^{n}\|_{\infty} \leq M.
\end{equation}
Indeed, if by contradiction we suppose that this is not true, then, since $z_{k}^{n}(\underline{\omega})=0$, we consider the maximal interval of the form $\mathopen{[}\underline{\omega},\nu_{n}\mathclose{]}$ such that $|z_{k}^{n}(x)| \leq M$ for all $x\in \mathopen{[}\underline{\omega},\nu_{n}\mathclose{]}$ and $|z_{k}^{n}(x)| > M$ for some $x\in \mathopen{]}\nu_{n},\overline{\omega}\mathclose{[}$.
By the maximality of the interval and the continuity of $z_{k}^{n}$, we also know that $|z_{k}^{n}(\nu_{n})| = M$.
Integrating \eqref{eq-zn} on $\mathopen{[}\underline{\omega},\nu_{n}\mathclose{]}$ and passing to the absolute value, we obtain
\begin{equation*}
\begin{aligned}
M &= |z_{k}^{n}(\nu_{n})| = |z_{k}^{n}(\nu_{n}) - z_{k}^{n}(\underline{\omega})| 
\leq \Bigl{|} \int_{\underline{\omega}}^{\nu_{n}} f_{k}'(p_{k}^{n}(x)) (z_{k}^{n}(x))^{2} \,\mathrm{d}x \Bigr{|} + \vartheta_{n}\|\gamma_{k}\|_{L^{1}}\\
& \leq \delta M^{2} |\nu_{n} - \underline{\omega}| + \|\gamma_{k}\|_{L^{1}}
\leq \delta M^{2} |\Omega| + \|\gamma_{k}\|_{L^{1}},
\end{aligned}
\end{equation*}
a contradiction with the choice of $\delta$ in \eqref{eq-delta}. In this manner, we have verified that \eqref{eqz-ineq} is true.

Now, integrating \eqref{eq-zn} on $\Omega$ and using \eqref{eqz-ineq}, we obtain that
\begin{equation*}
\begin{aligned}
0< - \lambda_{k} \int_{\Omega} \beta_{k}(x)\,\mathrm{d}x &\leq - \lambda_{k} \int_{\Omega} w_{k}(x,\widehat{p^{n}}^{k}(x))\,\mathrm{d}x \\
&= \vartheta_{n} \int_{\Omega} f_{k}'(p_{k}^{n}(x)) (z_{k}^{n}(x))^{2} \,\mathrm{d}x\\
&\leq M^{2}|\Omega| \max_{0\leq s \leq r^{n}}|f_{k}'(s)|
\end{aligned}
\end{equation*}
holds for every $n$ sufficiently large. Using the continuity of $f_{k}'$ at $0^{+}$, we get a contradiction, as $n\to \infty$.
\end{proof}

\begin{lemma}\label{lem-R}
For $i=1,\ldots,N$, let $w_{i}\colon\Omega\times\mathbb{R}^{N-1}\to\mathbb{R}$ be an $L^{1}$-Carath\'{e}odory function satisfying \ref{cond:wi*}. For $i=1,\ldots,N$, let $f_{i}\colon\mathopen{[}0,1\mathclose{]}\to\mathopen{[}0,+\infty\mathclose{[}$ be a continuously differentiable function satisfying \ref{cond:fi*}. 
Let $\lambda_{i}>0$, for $i=1,\ldots,N$. For every $k\in\{1,\ldots,N\}$ there exists $R_{k}\in\mathopen{]}0,1\mathclose{[}$ such that for all $\vartheta\in\mathopen{]}0,1\mathclose{]}$ every solution $p=(p_{1},\ldots,p_{N})$ of \eqref{system-S_theta} satisfies $\|p_{k}\|_{\infty}<R_{k}$.
\end{lemma}

\begin{proof}
Let $k\in\{1,\ldots,N\}$ be fixed.
By contradiction, we assume that there exists a sequence $(p^{n})_{n}$ of solutions of \eqref{system-S_theta} for $\vartheta=\vartheta_{n}\in\mathopen{]}0,1\mathclose{]}$, with $p^{n}=(p_{1}^{n},\ldots,p_{N}^{n})$ such that $0<\|p_{k}^{n}\|_{\infty}=R^{n}\to 1^{-}$ as $n\to\infty$.

As a first step, we claim that $p_{k}^{n}\to1$ uniformly as $n\to\infty$. 
By the uniqueness of the solution of the Cauchy problem associated with the $k$-th equation in \eqref{system-S_theta}, since $p_{k}^{n}\not\equiv1$, we notice that $(1-p_{k}^{n}(x))^{2}+((p_{k}^{n})'(x))^{2}>0$ for all $x\in\overline{\Omega}$.
Since $f_{k}\in\mathcal{C}^{1}(\mathopen{[}0,1\mathclose{]})$, we can fix a constant $C>0$ such that $f_{k}(s)\leq C(1-s)$ for every $s\in\mathopen{[}0,1\mathclose{]}$. Let $\gamma_{k}=\lambda_{k} \max\{|\alpha_{k}|,|\beta_{k}|\}$. 
Therefore, we have
\begin{align*}
&\biggl{|}\frac{\mathrm{d}}{\mathrm{d}x} \log{\bigl{(}(1-p_{k}^{n}(x))^{2}+((p_{k}^{n})'(x))^{2}\bigr{)}}\biggr{|} =\\
&=\biggl{|}-2\;\frac{(1-p_{k}^{n}(x))(p_{k}^{n})'(x)+\vartheta_{n} \lambda_{k} w_{k}(x,\widehat{p^{n}}^{k}(x)) (p_{k}^{n})'(x)f_{k}(p_{k}^{n}(x))}{(1-p_{k}^{n}(x))^{2}+((p_{k}^{n})'(x))^{2}}\biggr{|}\\
&\leq2\frac{(1-p_{k}^{n}(x))|(p_{k}^{n})'(x)|+\gamma_{k}(x) |(p_{k}^{n})'(x)| f_{k}(p_{k}^{n}(x))}{(1-p_{k}^{n}(x))^{2}+((p_{k}^{n})'(x))^{2}}\\
&\leq (1+C \gamma_{k}(x))\frac{2(1-p_{k}^{n}(x))|(p_{k}^{n})'(x)|}{(1-p_{k}^{n}(x))^{2}+((p_{k}^{n})'(x))^{2}}\\
&\leq 1+C \gamma_{k}(x), \quad \text{for all $x\in\Omega$}.
\end{align*}
Let $\bar{x}_{n}\in\Omega$ be such that $p_{k}^{n}(\bar{x}_{n})=R^{n}$.
Hence, by an integration of the above inequality from $\bar{x}_{n}$ to an arbitrary $x\in \Omega$, we have 
\begin{equation*}
\log{\frac{(1-p_{k}^{n}(x))^{2} +((p_{k}^{n})'(x))^{2}}{(1-R^{n})^{2} }}\leq |\Omega| + C\|\gamma_{k}\|_{L^{1}}
\end{equation*}
and so we deduce
\begin{equation*}
(1-p_{k}^{n}(x))^{2} +((p_{k}^{n})'(x))^{2}\leq (1-R^{n})^{2}e^{|\Omega|+C\|\gamma_{k}\|_{L^{1}}},
\end{equation*}
for all $x\in\overline{\Omega}$. The claim is thus proved since $R^{n}\to1^{-}$ as $n\to\infty$.

As in Lemma~\ref{lem-r} we perform the change of variable \eqref{eq-zeta-n} and we integrate \eqref{eq-zn} on $\Omega$, obtaining
\begin{equation}\label{eq-cc}
0 < -\lambda_{k} \int_{\Omega} \beta_{k}(x) \,\mathrm{d}x \leq \vartheta_{n} \int_{\Omega} f_{k}'(p_{k}^{n}(x))(z_{k}^{n}(x))^{2}\,\mathrm{d}x.
\end{equation}
If $f_{k}'(1)<0$, then $f_{k}'(s)<0$ for every $s$ in a left neighbourhood of $1$ and a contradiction follows from \eqref{eq-cc} since $p_{k}^{n}\to1$ uniformly as $n\to\infty$. On the other hand, if $f_{k}'(1)=0$, by arguing as in Lemma~\ref{lem-r}, the sequence $(z_{k}^{n})_{n}$ is uniformly bounded, $f_{k}'(p_{k}^{n})\to 0$ uniformly as $n\to\infty$, and a contradiction is reached from \eqref{eq-cc}.
\end{proof}

\subsection{Proof of Theorem~\ref{th-main}}\label{section-4.2}
Let $\rho\in\mathopen{]}0,1\mathclose{[}$ be arbitrarily fixed. For every $i=1,\ldots,N$, let $\lambda^{*}_{i} =\lambda^{*}_{i} (\rho) > 0$ be the constant given by  Lemma~\ref{lem-rho}. Then, we define
\begin{equation*}
\lambda^{*}=\max_{i=1,\ldots,N} \lambda^{*}_{i}
\end{equation*}
and fix
\begin{equation*}
\lambda_{i} > \lambda^{*}, \quad \text{for every $i=1,\ldots,N$.}
\end{equation*}

Next, we apply Lemma~\ref{lem-r} and Lemma~\ref{lem-R}, obtaining $2N$ constants $r_{i},R_{i}\in\mathopen{]}0,1\mathclose{[}$. Without loss of generality we assume $0<r_{i} < \rho < R_{i} < 1$, for every $i=1,\ldots,N$. Then, we define
\begin{equation*}
r=\min_{i=1,\ldots,N} r_{i} \quad\text{ and }\quad R=\max_{i=1,\ldots,N} R_{i},
\end{equation*}
and so have $0<r<\rho<R<1$.

We are going to compute the coincidence degrees of $L$ and $N$ in the open sets
\begin{equation*}
\mathcal{B}_{d}^{\mathcal{I}} = \prod_{i=1}^{N} B_{i}, \quad \text{with $d=(d_{1},\ldots,d_{N})\in\mathopen{]}0,1\mathclose{[}^{N}$,}
\end{equation*}
where $\mathcal{I}\subseteq \{1,\ldots,N\}$ is a subset of indices, $d_{i}\in\{r,R\}$ if $i\in\{1,\ldots,N\}\setminus\mathcal{I}$, $d_{i}=\rho$ if $i\in\mathcal{I}$, and
\begin{equation*}
B_{i}= 
\begin{cases}
\, B(0,d_{i}), & \text{if $i\notin\mathcal{I}$} \\
\, U_{\rho,I_{i}}, & \text{if $i\in\mathcal{I}$.}
\end{cases}
\end{equation*}
Precisely, we prove that
\begin{equation}\label{eq-deg-B}
\mathrm{D}_{L}(L-N,\mathcal{B}_{d}^{\mathcal{I}}) = 
\begin{cases}
\, 1, & \text{if $\mathcal{I}=\emptyset$,} \\
\, 0, & \text{if $\mathcal{I}\neq\emptyset$.}
\end{cases}
\end{equation}

Firstly, we consider the set $\mathcal{B}_{d}^{\emptyset}$ with $d_{i}\in\{r,R\}$ for every index $i=1,\ldots,N$. We are going to verify that condition \ref{cond-H3} of Lemma~\ref{lem-deg1} is satisfied.
Let $\vartheta\in\mathopen{]}0,1\mathclose{]}$ and suppose, by contradiction, that $p=(p_{1},\ldots,p_{N})$ is a solution of \eqref{system-S_theta} such that at least one of the following cases occurs:
\begin{itemize}
\item there exists an index $k\in\{1,\ldots,N\}$ such that $\|p_{k}\|_{\infty}=r$,
\item there exists an index $k\in\{1,\ldots,N\}$ such that $\|p_{k}\|_{\infty}=R$.
\end{itemize}
In the first case, an application of Lemma~\ref{lem-r} gives that $p_{k}(x)=0$ for all $x\in\overline{\Omega}$, and thus a contradiction. The second case contradicts Lemma~\ref{lem-R}.

Secondly, we consider the set $\mathcal{B}_{d}^{\mathcal{I}}$ with $\mathcal{I}\neq\emptyset$. We are going to verify conditions \ref{cond-H1} and \ref{cond-H2} of Lemma~\ref{lem-deg0}, with the choice $v=(v_{1},\ldots,v_{N})$ where $v_{i}\equiv0$ if $i\notin\mathcal{I}$ and $v_{i}\equiv 1$ if $i\in\mathcal{I}$.
Let $\vartheta\in\mathopen{]}0,1\mathclose{]}$. In order to verify condition \ref{cond-H2} we suppose, by contradiction, that $p=(p_{1},\ldots,p_{N})$ is a solution of \eqref{system-S_theta} such that at least one of the following cases occurs:
\begin{itemize}
\item there exists an index $k\in\{1,\ldots,N\}\setminus\mathcal{I}$ such that $\|p_{k}\|_{\infty}=r$,
\item there exists an index $k\in\mathcal{I}$ such that $\max_{x\in I_{k}} |p_{k}(x)| = \rho$,
\item there exists an index $k\in\{1,\ldots,N\}\setminus\mathcal{I}$ such that $\|p_{k}\|_{\infty}=R$.
\end{itemize}
A contradiction is obtained by applying Lemma~\ref{lem-r} or Lemma~\ref{lem-R} (with $\vartheta=1$) if $k\notin\mathcal{I}$ and Lemma~\ref{lem-rho} if $k\in\mathcal{I}$. 
Next, we integrate the $k$-th equation in \eqref{system-Fa} and pass to the absolute value to obtain
\begin{equation*}
\mu |\Omega| \leq \| \gamma_{k} \|_{L^{1}} \max_{s \in \mathopen{[}0,1\mathclose{]}} f_{k}(s), \quad \text{with $\gamma_{k}=\lambda_{k} \max\{|\alpha_{k}|,|\beta_{k}|\}$.}
\end{equation*}
Therefore, condition \ref{cond-H3} follows for $\mu_{0}= \max \{\| \gamma_{i} \|_{L^{1}} \max_{s \in \mathopen{[}0,1\mathclose{]}} f_{i}(s) / |\Omega| \colon i=1,\ldots,N\}$.

Having proved that formula \eqref{eq-deg-B} holds, for $\mathcal{I},\mathcal{J}\subseteq \{1,\ldots,N\}$ with $\mathcal{I}\cap\mathcal{J}=\emptyset$, we introduce the sets
\begin{equation*}
\mathcal{A}^{\mathcal{I},\mathcal{J}}= \prod_{i=1}^{N} A_{i}, 
\end{equation*}
where
\begin{equation*}
A_{i}= 
\begin{cases}
\, B(0,r), & \text{if $i\in\{1,\ldots,N\}\setminus(\mathcal{I}\cup\mathcal{J})$,} \\
\, U_{\rho,I_{i}}\setminus B(0,r), & \text{if $i\in\mathcal{I}$,} \\
\, B(0,R)\setminus U_{\rho,I_{i}}, & \text{if $i\in\mathcal{J}$.}
\end{cases}
\end{equation*}
The sets $\mathcal{A}^{\mathcal{I},\mathcal{J}}$ are Cartesian products of ``annuli'' in $\mathcal{C}(\overline{\Omega})$, are pairwise disjoint, and, considering all the possible choices of $\mathcal{I}$ and $\mathcal{J}$, their number is $3^{N}$.

As a consequence of a powerful combinatorial argument illustrated in \cite[Appendix~A.1]{BoFeZa-18tams} and exploited in a different framework in the same paper, from \eqref{eq-deg-B} and the additivity property of the degree we deduce that
\begin{equation}\label{eq-deg-A}
\mathrm{D}_{L}(L-N,\mathcal{A}^{\mathcal{I},\mathcal{J}}) = (-1)^{\#\mathcal{I}}.
\end{equation}
We refer to Remark~\ref{rem-4.1} for the explicit derivation of \eqref{eq-deg-A} in the case $N=2$.

Formula \eqref{eq-deg-A} and the existence property of the coincidence degree imply that in each of the sets of the form $\mathcal{A}^{\mathcal{I},\mathcal{J}}$ there exists a solution $p=(p_{1},\ldots,p_{N})$ of \eqref{system-H}.
We observe that, by the maximum principle, it holds that $p_{i}(x) \geq 0$ for all $x\in\overline{\Omega}$, for every $i=1,\ldots,N$, and, moreover, by the definition of $\mathcal{A}^{\mathcal{I},\mathcal{J}}$, we clearly have $p_{i}(x) <1$ for all $x\in\overline{\Omega}$, for every $i=1,\ldots,N$. Hence, $p$ is a solution of \eqref{system-S}.

As a final step, we distinguish between constant solutions, semitrivial solutions and fully nontrivial solutions. We remark that if $\mathcal{I}\cup\mathcal{J}=\{1,\ldots,N\}$ then $A_{i}$ is the set $U_{\rho,I_{i}}\setminus B(0,r)$ or the set $B(0,R)\setminus U_{\rho,I_{i}}$, for every $i=1,\ldots,N$. In such a situation, the solution $p$ in $\mathcal{A}^{\mathcal{I},\mathcal{J}}$ is such that
$0 < r < \|p_{i}\|_{\infty} < R$, for every $i=1,\ldots,N$.
Now, the uniqueness of the constant zero solution for the Cauchy problem associated with \eqref{system-S} is ensured by the fact that the functions $f_{i}$ are of class $\mathcal{C}^1$; therefore, by the strong maximum principle, we deduce that $p$ is a fully nontrivial solution of \eqref{system-S}. The number of partitions of $\{1,\ldots,N\}$ in two sets $\mathcal{I},\mathcal{J}$ with $\mathcal{I}\cap\mathcal{J}=\emptyset$ and $\mathcal{I}\cup\mathcal{J}=\{1,\ldots,N\}$ is $2^{N}$, since, for every $i=1,\ldots,N$, we have two choices
\begin{itemize}
\item[$(i)$] $r<\max_{x\in I_{i}} |p_{i}(x)| < \rho$,
\item[$(ii)$] $\rho<\max_{x\in I_{i}} |p_{i}(x)| < R$.
\end{itemize}
In conclusion, there are $2^{N}$ fully nontrivial solutions of \eqref{system-S}.

In order to count the trivial solutions, we list the two possibilities for each component
\begin{itemize}
\item[$(iii)$] $p_{i}\equiv 0$,
\item[$(iv)$] $p_{i}\equiv 1$,
\end{itemize}
obtaining that their number is $2^{N}$.
At last, in order to count the semitrivial solutions, we have to consider the four possibilities $(i)$--$(iv)$ for each component and thus their number is $4^{N}-2^{N}$.
The proof is thus concluded.
\qed

\begin{remark}\label{rem-4.1}
In order to clarify the derivation of formula \eqref{eq-deg-A}, we now give the direct computation of the degree in the case $N=2$, without using the combinatorial argument developed in \cite{BoFeZa-18tams}. 
As an example we compute the degree in $A^{\emptyset,\{1\}}$.
Starting from
\begin{equation}\label{eq-exc}
A^{\emptyset,\{1\}} = \bigl{(} B(0,R) \times B(0,r) \bigr{)} \setminus \overline{\bigl{[} A^{\emptyset,\emptyset} \cup A^{\{1\},\emptyset} \bigr{]}},
\end{equation}
our goal is to apply the additivity property of the coincidence degree. Accordingly we first compute $\mathrm{D}_{L}(L-N,\mathcal{A}^{\emptyset,\emptyset})$ and $\mathrm{D}_{L}(L-N,\mathcal{A}^{\{1\},\emptyset})$.
Since $\mathcal{A}^{\emptyset,\emptyset}=B(0,r)\times B(0,r)$, formula \eqref{eq-deg-B} implies that
\begin{equation*}
\mathrm{D}_{L}(L-N,\mathcal{A}^{\emptyset,\emptyset})=1.
\end{equation*}
Next, we observe that $A^{\{1\},\emptyset} = \bigl{(} U_{\rho,I_{1}} \times B(0,r) \bigr{)} \setminus \overline{A^{\emptyset,\emptyset}}$. From formula \eqref{eq-deg-B} and the excision property of the degree (noticing that there are no solutions of \eqref{system-S} on the boundary of the sets $\mathcal{A}^{\emptyset,\emptyset}$ and $\mathcal{A}^{\{1\},\emptyset}$, as already shown) we obtain
\begin{equation*}
\mathrm{D}_{L}(L-N,\mathcal{A}^{\{1\},\emptyset})=\mathrm{D}_{L}(L-N,U_{\rho,I_{1}} \times B(0,r)) -\mathrm{D}_{L}(L-N,\mathcal{A}^{\emptyset,\emptyset}) =-1.
\end{equation*}
As a final step, arguing as above, from \eqref{eq-exc} and the additivity of the degree, we have
\begin{align*}
&\mathrm{D}_{L}(L-N,A^{\emptyset,\{1\}})=\\
&=\mathrm{D}_{L}(L-N,B(0,R)\times B(0,r))-\mathrm{D}_{L}(L-N,\mathcal{A}^{\emptyset,\emptyset}) - \mathrm{D}_{L}(L-N,\mathcal{A}^{\{1\},\emptyset}) 
\\&=1-1-(-1)=1.
\end{align*}

Proceeding as above, one can obtain the values of the degree in all the remaining sets of the form $A^{\mathcal{I},\mathcal{J}}$, namely
\begin{align*}
&\mathrm{D}_{L}(L-N,\mathcal{A}^{\{2\},\emptyset})=-1,
&&\mathrm{D}_{L}(L-N,\mathcal{A}^{\{1,2\},\emptyset})=1,
\\
&\mathrm{D}_{L}(L-N,\mathcal{A}^{\{1\},\{2\}})=-1,
&&\mathrm{D}_{L}(L-N,\mathcal{A}^{\{2\},\{1\}})=-1,
\\
&\mathrm{D}_{L}(L-N,\mathcal{A}^{\emptyset,\{2\}})=1,
&&\mathrm{D}_{L}(L-N,\mathcal{A}^{\emptyset,\{1,2\}})=1.
\end{align*}
The four fully nontrivial solutions are cointained in the sets $\mathcal{A}^{\{1\},\{2\}}$, $\mathcal{A}^{\{2\},\{1\}}$, $\mathcal{A}^{\{1,2\},\emptyset}$, $\mathcal{A}^{\emptyset,\{1,2\}}$.

For a general integer $N$ one can prove formula \eqref{eq-deg-A} using the combinatorial argument in \cite{BoFeZa-18tams} or by induction using  the excision property of the degree.
\hfill$\lhd$
\end{remark}

\begin{remark}\label{rem-4.2}
From the proof, one can observe that Theorem~\ref{th-main} is still valid if the nonlinearities $f_{i}$ are assumed to be only continuous in $\mathopen{[}0,1\mathclose{]}$, and continuously differentiable in a right neighbourhood of $s = 0$ and in a left neighbourhood of $s=1$.
\hfill$\lhd$
\end{remark}

\section{Application to selection-migration models in population genetics}\label{section-5}

In this section we show how our abstract result applies to a selection-migration model in population genetics (cf.~\cite{Bu-00,Bu-14}).
In order to gradually introduce our motivating model, discussed in Section~\ref{section-5.2}, we first briefly recall the basics of selection-migration models, and then, in Section~\ref{section-5.1}, revise the case of a diploid population with selection only at one locus.

We study the continuous distribution of a given trait of a population on a bounded habitat, that we denote with $\Omega$. The evolution of the state of the population is guided by two main mechanisms: dispersal in the habitat and competition between genotypes.
Let us therefore denote with $p=p(t,x)\colon \mathbb{R}\times \Omega\to \mathopen{[}0,1\mathclose{]}$ the frequency of a trait $a$ at time $t$ and position $x$. The evolution of the population is described by the nonlinear reaction-diffusion equation:
\begin{equation}\label{eq:gen_reacdiff}
\frac{\partial p}{\partial t}= \frac{\kappa(x)}{\lambda} \Delta p+F(x,p).
\end{equation}
The Laplacian $\Delta$ describes the dispersal of the population, weighted by a coefficient $\kappa(x)>\kappa_{0}>0$. The parameter $\lambda>0$ controls the ratio between the rates of selection and diffusion. The term $F(x,p)$ accounts for the selection on the traits and can be expressed in the form 
\begin{equation}\label{eq:fitnessrate}
F(x,p)=p(r_{a}(x)-\bar r(x,p)),
\end{equation}
where $r_{a}$ is the fitness of the trait and $\bar r$ is the average fitness of the population, both taken at $x$. In other words, $r_{a}-\bar r$ is the relative fitness of the trait.
We observe that
\begin{align*}
F(x,0)=F(x,1)=0.
\end{align*}
This property means that if a trait is absent in the population, it will remain so; indeed, we assume no mutation, nor constant immigration rate from an external environment. 
As we will derive in the models we discuss below, in some circumstances it can be observed that the term $F(x,p)$ is factorised in the following way
\begin{equation} \label{eq:Fproduct}
F(x,p)=\omega(x)f(p).
\end{equation}
The factor $\omega(x)$  describes whether there is a competitive advantage ($\omega(x)>0$) or disadvantage ($\omega(x)<0$) for the trait with respect to the population in a specific place $x$ of the habitat. The factor $f(p)$ describes the frequency-dependent effects on the selection. 

Moreover, since we consider bounded habitats $\Omega$, we impose the condition of zero normal derivative at the boundary
\begin{equation}\label{eq:bvc}
\dfrac{\partial p}{\partial \nu}=0, \quad \text{on $\partial \Omega$,}
\end{equation}
meaning that there is no population flow at the boundary into or out of the habitat.

The most famous case of equation \eqref{eq:gen_reacdiff} is \emph{Fisher's equation}, corresponding to constant $\omega(x)=\omega$ and $\kappa(x)=\kappa$, and $f(p)=p(1-p)$. 
In the framework above, it corresponds to the competition between two types $a$ and $A$ with constant fitnesses $r_{a}$ and $r_{A}$, so that, by \eqref{eq:fitnessrate},
\begin{equation*}
F(x,p)=p \bigl{[} r_{a}-pr_{a}-(1-p)r_{A} \bigr{]}=p(1-p)(r_{a}-r_{A}).
\end{equation*}

We restrict our analysis to the case of a one-dimensional habitat. Our interest is focused on the search for stationary solution for \eqref{eq:gen_reacdiff}, corresponding to the solutions of 
\begin{equation*}
p''+\lambda\,\frac{\omega(x)}{\kappa(x)}f(p)=0
\end{equation*}
with the boundary condition \eqref{eq:bvc}. In particular, we are interested in \emph{fully nontrivial solution}, meaning that $0<p<1$, corresponding to steady states for the population where the two traits coexist. Such solutions are often called \emph{clines}. Indeed, the main motivation of the models we consider in the following is to investigate how a heterogeneous habitat (i.e.~a sign-changing $\omega(x)$) can be a mechanism for the preservation of polymorphism in the population.

\subsection{Diploid population with selection at one locus}\label{section-5.1}

 We now extend our discussion to the case of a diploid population with two alleles 
$a$ and $A$ (cf.~\cite{Fl-75}). Hence we consider not any longer the frequency of a trait, but instead the frequency of the allele $a$ in the population, which we denote by $p$. We assume random mating and that the allele frequencies are at equilibrium, so that the genotype distribution follows Hardy--Weinberg law.
We denote the fitness of the genotypes as
\renewcommand\arraystretch{1.5}
\begin{equation*}
\begin{array}{ccc}
aa & Aa & AA\\\hline
r_{aa}(x) & r_{Aa}(x) & r_{AA}(x)
\end{array}
\end{equation*}
and hence, recalling that $aa$-individuals account for two copies of the allele $a$ in the genetic pool, we replace \eqref{eq:fitnessrate} by
\begin{equation}\label{eq:reacElisa}
F(x,p)=2p ^{2}(r_{aa}(x)-\bar r(x,p))+2p(1-p)(r_{Aa}(x)-\bar r(x,p)),
\end{equation}
where
\begin{equation*}
\bar r(x,p)=p ^{2}r_{aa}(x)+2p(1-p)r_{Aa}(x)+(1-p) ^{2}r_{AA}(x).
\end{equation*}
Often it is assumed that $r_{Aa}(x)$ can expressed in terms of the other two rates using a parameter $h$, according to
\begin{equation*}
r_{Aa}(x)=\frac{r_{aa}(x)+r_{AA}(x)}{2}+h\,\frac{r_{aa}(x)-r_{AA}(x)}{2},
\end{equation*}
which gives $f(p)=p(1-p)(1+h-2hp)$, cf.~\cite{Fl-75}.
In this framework, two specific situations are remarkable. The first one is the case when $r_{Aa}(x)=\frac{1}{2}(r_{aa}(x)+r_{AA}(x))$, i.e.~$h=0$, which is equivalent to a space dependent Fisher's equation.
The second remarkable case occurs when we have complete dominance of an allele, which without loss of generality we assume to be $A$. In other words, we have $r_{Aa}(x)=r_{AA}(x)$, i.e.~$h=-1$.
In this case a straightforward computation shows that $F(x,p)$ can be expressed in the form \eqref{eq:Fproduct} with 
\begin{equation*}
f(p)=p ^{2}(1-p), \qquad \omega(x)=2(r_{aa}(x)-r_{AA}(x)).
\end{equation*}

As in the previous case, we are interested in mechanisms for the coexistence of the two alleles in the population. The simpler one is heterosis, namely assuming the heterozygote is fitter than either homozygote ($r_{Aa}>r_{aa}$ and $r_{Aa}>r_{AA}$). We are however interested in situations where polymorphism is preserved by the combination of migration and selection, and not by selection alone.
Following Fleming \cite{Fl-75} and Henry \cite{He-81} the classical approach to study such a situation is to assume that $F$ is of the form \eqref{eq:Fproduct} where $f$ is a $\mathcal{C}^{1}$-function satisfying
\begin{equation}\label{cond:flinear}
f(0)=f(1)=0, \qquad f'(0)>0>f'(1), \qquad f(s)>0, \quad \text{for all $s\in\mathopen{]}0,1\mathclose{[}$,}
\end{equation}
whereas $\omega(x)$ is sign-changing. This framework covers the cases corresponding to $h\in \mathopen{]}-1,1\mathclose{[}$ and has been extensively studied, showing different behaviour according to the sign of the average of $\omega(x)$.
We mention for instance \cite{LoNa-02,LoNaNi-13,Na-89} and refer to \cite{NaSuDu-19,So-18} for recent surveys.

However \eqref{cond:flinear} leaves out the case of complete dominance $h=-1$, where we have $f'(0)=0$ (or $f'(1)=0$, if we consider the symmetric case $h=1$), which is indeed the superlinear framework we introduced with conditions \ref{cond:fi*} and \ref{cond:fi0}.

Such a superlinear behaviour in the origin has proven to be suitable to the application of analytical tools. Within the standing assumptions, it has been shown that in an overall hostile environment ($\omega(x)$ with negative average) with a sufficiently intense selection ($\lambda$ large) we have at least two nontrivial solutions \cite{BoFeSo-PP,NaNiSu-10}.
Moreover, using topological degree \cite{BoFeSo-PP} or shooting methods \cite{FeSo-18non,FeSo-18na}, the existence of multiple disconnected patches of favourable habitat (i.e.~the existence of multiple disjoint subintervals of $\Omega$ where $\omega(x)$ is positive) produces an increasingly large number of clines.
In the specific case $f(p)=p^{2}(1-p)$ some stability results for the clines have also been obtained \cite{LoNiSu-10}.

\subsection{Diploid population with selection at two loci}\label{section-5.2}

We now extend the previous scenario to a multilocus model. 
Therefore, we have to discuss first in which way we expect the state at one locus to influence the evolution at the other one.
One possibility considered in literature to couple the selection processes at the two (or more) loci is \emph{linkage disequilibrium}. It has been shown that linkage disequilibrium produces a steepening of the cline \cite{Bu17,Sl75} and that such deviation by a Hardy--Weinberg distribution has little evolutionary significance \cite{Na-76}. We also mention \cite{SLB19} for stability results. In all these works linkage disequilibrium is the only coupling effect, since additivity of the fitnesses of the two loci is assumed (namely, the effects on the fitness of the individual produced by the genotypes at different loci are independent of each other).
Our approach is complementary. We assume linkage equilibrium, but allow general non-additive interactions between the fitnesses of the genotypes at different loci.  In simpler words, we assume that genotypes at different loci may produce a combined effect on the fitness of the individual is greater than the sum of their single effects. A general treatment of this case in the weak-selection limit (hence neglecting the spatial dependence) has been recently presented in \cite{PHB18}.

In more details, we consider a diploid population with two alleles $a$ and $A$ at one locus and other two alleles $b$ and $B$ at another locus. We assume complete dominance at both loci, with the alleles $a$ and $b$ being the recessive ones. We denote with $p$ and $q$ the frequencies of the alleles $a$ and $b$ in the population. We assume random mating, linkage equilibrium between the two loci, and that the allele frequencies are at equilibrium, so that Hardy--Weinberg law applies independently to each locus, namely
\begin{equation*}
\begin{array}{r|ccc}
\text{\tiny\shortstack{ genotype \\ frequency}}&aa & Aa & AA\\\hline
bb& p^{2}q^{2} & 2p(1-p)q^{2} & (1-p)^{2}q^{2}\\
Bb& 2p^{2}q(1-q) & 4p(1-p)q(1-q) & 2(1-p)^{2}q(1-q)\\
BB&	p^{2}(1-q)^{2} & 2p(1-p)(1-q)^{2} & (1-p)^{2}(1-q)^{2}\\
\end{array}
\end{equation*}
Recalling that with dominance at both loci we have only four possible phenotypes, we denote the fitness of the genotypes as
\begin{equation*}
\begin{array}{r|ccc}
\text{\tiny\shortstack{ genotype \\ fitness}}&aa & Aa & AA\\\hline
bb& r_{ab}(x) & r_{Ab}(x) & r_{Ab}(x)\\
Bb& r_{aB}(x) & r_{AB}(x) & r_{AB}(x)\\
BB&	r_{aB}(x) & r_{AB}(x) & r_{AB}(x)\\
\end{array}
\end{equation*}
(to simplify the notation in the subscripts, we are using a lowercase letter when the recessive allele is expressed, and an uppercase one when the dominant is expressed).

Proceeding analogously to \eqref{eq:reacElisa} for both alleles, and recalling that the fitness of an allele depends also by the allele distribution in the other locus, we can recover the fitnesses of the alleles $a$ and $b$, namely
\begin{align*}
F_{a}(x,p,q)&=2\bigl{[}q ^{2}\bigl(r_{ab}(x)-r_{Ab}(x)-r_{aB}(x)+r_{AB}(x)\bigr)
\\ &\qquad+r_{aB}(x)-r_{AB}(x)\bigr{]}p ^{2}(1-p),\\
F_{b}(x,p,q)&=2\bigl[p ^{2}\bigl(r_{ab}(x)-r_{Ab}(x)-r_{aB}(x)+r_{AB}(x)\bigr)
\\ & \qquad+r_{Ab}(x)-r_{AB}(x)\bigr]q ^{2}(1-q).
\end{align*}
Let us write
\begin{align*}
\omega^{ab}(x)&=r_{ab}(x)-r_{Ab}(x), &
\omega^{aB}(x)&=r_{aB}(x)-r_{AB}(x), \\
\omega^{ba}(x)&=r_{ab}(x)-r_{aB}(x), &
\omega^{bA}(x)&=r_{Ab}(x)-r_{AB}(x).
\end{align*}
This means, for instance, that $\omega^{aB}$ measures the advantage (if positive) or disadvantage (if negative) of the homozygous-recessive $aa$ versus the homozygous-dominant $AA$ (or equivalently versus the heterozygous $Aa$) when they are both expressed coupled with the homozygous $BB$ (or equivalently versus the heterozygous $Bb$).
We also set
\begin{equation}\label{eq:fdomin}
f(s)=2s ^{2}(1-s).
\end{equation}
Hence the steady states of the population correspond to the solutions of the Neumann problem associated with the system
\begin{equation}\label{sys-intro}
\begin{cases}
\, p''+\lambda\,\dfrac{q ^{2}\omega^{ab}(x)+(1-q ^{2})\omega^{aB}(x)}{\kappa(x)}\,f(p)=0, \vspace{5pt}\\
\, q''+\lambda\,\dfrac{p ^{2}\omega^{ba}(x)+(1-p ^{2})\omega^{bA}(x)}{\kappa(x)}\,f(q)=0.
\end{cases}
\end{equation}
Let us define
\begin{align*}
\alpha_{a}(x)=\frac{\min \{\omega^{ab}(x),\omega^{aB}(x)\}}{\kappa(x)},
&& \beta_{a}(x)=\frac{\max \{\omega^{ab}(x),\omega^{aB}(x)\}}{\kappa(x)},\\
\alpha_{b}(x)=\frac{\min \{\omega^{ba}(x),\omega^{bA}(x)\}}{\kappa(x)},
&& \beta_{b}(x)=\frac{\max \{\omega^{ba}(x),\omega^{bA}(x)\}}{\kappa(x)},
\end{align*}
and
\begin{align*}
&w_{p}(x,q)=\dfrac{q ^{2}\omega^{ab}(x)+(1-q ^{2})\omega^{aB}(x)}{\kappa(x)},\\
&w_{q}(x,p)=\dfrac{p ^{2}\omega^{ba}(x)+(1-p ^{2})\omega^{bA}(x)}{\kappa(x)}.
\end{align*}
In order to satisfy \ref{cond:wi*}, for $i=p,q$, we require:
\begin{enumerate}[label=\textup{$(\roman*)$}]
\item there exists two intervals $I_{p},I_{q}\subseteq \Omega$ such that $\alpha_{a}(x)>0$ for every $x\in I_{p}$, and $\alpha_{b}(x)>0$ for every $x\in I_{q}$; \label{cond:modalpha}
\item $\displaystyle \int_\Omega \beta_{a}(x)\,\mathrm{d} x<0$ and $\displaystyle \int_\Omega \beta_{b}(x)\,\mathrm{d} x<0$. \label{cond:modbeta}
\end{enumerate}
These conditions on the sign of the weights means that in some place the environment is favourable, or not, to the recessive homozygous in one locus not regarding of the genotype in the other locus.

We have the following straightforward corollary of Theorem~\ref{th-main}.

\begin{corollary}\label{cor-intro}
Let $\Omega$ be an open bounded interval, $\kappa\colon \Omega\to \mathopen{[}\kappa_{0},+\infty\mathclose{[}$, with $\kappa_{0}>0$, be a Lebesgue integrable positive function, and $f\colon \mathopen{[}0,1\mathclose{]}\to \mathbb{R}$ be defined as in \eqref{eq:fdomin}. Assume that the functions $r_{ab},r_{Ab},r_{aB},r_{AB}\colon \Omega\to \mathbb{R}$ are Lebesgue integrable on $\Omega$ and satisfy \ref{cond:modalpha} and \ref{cond:modbeta}.
Then, there exists $\lambda^{*}>0$ such that for every $\lambda>\lambda^{*}$ there exist at least four fully nontrivial solutions of the Neumann problem associated with \eqref{sys-intro}.
\end{corollary}

\begin{remark}\label{rem:multilocus}
We remark that the same construction applies if we consider a model with $N$ loci, instead of two, with Corollary \ref{cor-intro} providing the existence of at least $2^N$ fully nontrivial solutions.
\hfill$\lhd$
\end{remark}

\section*{Acknowledgements}

We are grateful to Elisa Sovrano for her useful remarks on improving the presentation of the paper.

\bibliographystyle{elsart-num-sort}
\bibliography{FeGi-biblio}

\end{document}